\documentclass{amsart}

\usepackage[paperwidth=210mm,paperheight=297mm,hmargin={35mm,35mm},vmargin={35mm,35mm}]{geometry} 
\usepackage{amsthm,amssymb,latexsym,amsmath,mathbbol}
\usepackage{mathrsfs}
\usepackage{graphicx}
\usepackage{hyperref}
\input xy
\xyoption{all}
\usepackage[all]{xy}

\usepackage{amsthm,amssymb,latexsym,amsmath}
\usepackage{mathrsfs}
\usepackage{graphicx}
\usepackage{color}
\usepackage[latin1]{inputenc}
\usepackage[all]{xy}

\renewcommand{\dim}{\mathrm{dim}}

\newcommand{\OO}{{\mathcal{O}}}

\newcommand{\Sing}{{\rm Sing}}
\newcommand{\codim}{{\rm codim}}

\newtheorem{nada}{Nada}[section]

\newtheorem{proposition}[nada]{Proposition}
\newtheorem{corollary}[nada]{Corollary}
\newtheorem{thm}[nada]{Theorem}
\newtheorem*{thm*}{Theorem}
\newtheorem{lemma}[nada]{Lemma}

\newtheorem{rmk}[nada]{Remark}
\newtheorem{example}[nada]{Example}

\newcommand{\bc}{\begin{center}}
\newcommand{\ec}{\end{center}}
\newcommand{\noi}{\noindent}

\theoremstyle{plain}

\begin{document}

\title{A note of characteristic class for singular varieties}
\hyphenation{ho-mo-lo-gi-cal}
\hyphenation{fo-lia-tion}

\begin{abstract}

In this work we study characteristic classes of possibly singular varieties embedded as a closed subvariety of a nonsingular variety. In special, we express the Schwartz-MacPherson class in terms of the $\mu$-class and Chern class of the sheaves of logarithmic and multi-logarithmic differential forms. As an application we show an expression for Euler characteristic of a complement of a singular variety. \\

\noindent \textit{Keywords:} Characteristic class; Logarithmic sheaf; Milnor number; Chern class.

\end{abstract}

\author{Antonio M. Ferreira}
\address{ Antonio Marcos Ferreira da Silva \\ DMA - UFES,  Rodovia BR-101, Km 60, Bairro Litor\^aneo, S\~ao Mateus-ES, Brazil, CEP 29932-540}
\email{antonio.m.silva@ufes.br}

\author{Fernando Louren\c co}
\address{ Fernando Louren\c co \\ DMM - UFLA ,  Campus Universit\'ario, Lavras MG, Brazil, CEP 37200-000}
\email{fernando.lourenco@ufla.br}

\thanks{ }
\subjclass{Primary 32S65, 37F75; secondary 14F05}

\maketitle

\section{Introduction}

Let $X$ be a possibly singular variety in a nonsingular variety $M$. When $X$ is regular, there is a well know notion of the characteristic class of $X$, the Chern class, defined in many ways under its tangent bundle, see for instance, \cite{Fulton, Hart}. For the original definition see \cite{Chern}. On the other hand, if $X$ is singular one has different ways to generalize this characteristic class of $X$.

The first extension of Chern classes for a possibly singular variety is the Schwartz-MacPherson class, it was done independently by M.-H. Schwartz in 1965, see \cite{BrSch,schw} as an element on the cohomology group and by R. MacPherson in 1974 in \cite{MacPherson} as an element of the homology group.

Another important class with which we going to work in this paper is the Fulton class, denoted by $c_{F}(X)$, see \cite{Fulton}. This class is defined over a scheme $X$ that can be embedded in a nonsingular variety $M$ and it is proved that it is independent of the choice of embedded. An advantage of this class is that it can be defined over arbitrary fields and in a completely algebraic fashion. On the other hand, a disadvantage it is does not satisfy at first sight nice functorial properties.

The Milnor number was initially defined by Milnor in 1968,  see \cite{Mil}, to a hypersurface with isolated singularities. After, independently in 1971 Hamm \cite{Ham} and Lê in 1974 \cite{Le} extended this index for local complete intersections still with isolated singular points. More recently, in 1988, see \cite{Paru}, Parusinski extended the notion of Milnor number to nonisolated singularities. P. Aluffi in 1995, see \cite{Aluffi1},  defined another important class for the singular set of a hypersurface, the $\mu-$class, which in the case of isolated singularities coincides with the Minor number.

In \cite{AluBr} Aluffi et al., proved a formula relating the Mather
class, Schwartz-MacPherson class, and the class of the virtual tangent bundle of a hypersurface in a nonsingular variety, under certain assumptions in the singular locus of a hypersurface $X$.

Using the computer, Aluffi in \cite{Alu3} developed an algorithm computing to calculate some characteristic classes. The program computed the push-forward to $\mathbb{P}^{n}$ of the Schwartz-MacPherson class, and the Fulton class, and showed as is well known, that the Euler
characteristic equals the degree of the $0$-dimension component of the Schwartz-MacPherson.

In general, it is very hard to calculate the Chern class to singular variety, and there are few results in this direction. On the other hand, as suggested by J.-P. Brasselet, P. Aluffi in \cite{Aluffi} (see    Theorem \ref{thm 1.1} below) shows an expression to the Schwartz-MacPherson class by using the sheaf of logarithmic differential forms.

The guiding theme of this paper is to express a way of to calculate the Schwartz-MacPherson class in terms of the sheaf of logarithmic and multi-logarithmic differential forms with poles in certain varieties, when this sheaf is not locally free. For this, we use the comparison between the Schwartz-MacPherson and Fulton classes for hypersurface due P. Aluffi, see Theorem \ref{thm2.1}.
When the variety $X$ is a local complete intersection, the difference between these classes is called Milnor class, see (\cite{Brasselet}, Definition 1, p. 46).

Let $X$ be an algebraic variety over an algebraically closed field of characteristic zero. Assume that $X$ is embedded as a closed subvariety of a nonsingular variety $M$, by $i : X \rightarrow M.$

\begin{thm}[\cite{Aluffi}, Theorem 1]\label{thm 1.1} Let $X$ be as above. Let $ \pi : \tilde{M} \rightarrow M$ be a birational map with $\tilde{M}$ a nonsingular variety, such that $X^{'} = \Big(\pi^{-1}(X)\Big)_{red}$ is a divisor with smooth components and normal crossing in $\tilde{M}$, and $\pi|_{\tilde{M}-X^{'}}$ is an isomorphism. Then

$$ i_{\ast}c_{SM}(X) = c(TM)\cap [M]- \pi_{\ast} \Big( c(\Omega^{1}_{\tilde{M}}(\log X^{'})^{\vee})\cap [\tilde{M}]\Big)  \in A_{\ast}M.$$

\end{thm}



Now, follows our main result:

\begin{thm}\label{thm 1.2} Let $i: X \longrightarrow M$ be an embedding of a closed subvariety $X$ in a complex algebraic nonsingular variety $M$. Let $\pi :\tilde{M} \longrightarrow M$ be a proper birational map with $\tilde{M}$ a nonsingular variety, such that $(\pi^{-1}(X))_{red}= \tilde{X}$ is a hypersurface with its singular scheme denoted by $\tilde{Y}$ and $\pi|_{\tilde{M}\setminus \tilde{X}}$ is an isomorphism. We assume that $\codim_{\tilde{M}}(\tilde{Y})\geq 3$. Then

\begin{equation}\label{01}
i_{\ast}c_{SM}(X) = c(TM)\cap [M]- \pi_{\ast} \Big( c(\Omega^{1}_{\tilde{M}}(\log \tilde{X})^{\vee})\cap [\tilde{M}]\Big) + \pi_{\ast}j_{\ast}c(\mathcal{L})^{\dim \tilde{X}} \cap \Big(\mu_{\mathcal{L}}(\tilde{Y})^{\vee} \otimes_{\tilde{M}} \mathcal{L}\Big),
\end{equation}

\noi where $\mu_{\mathcal{L}}(\tilde{Y})$ denotes the $\mu$-class of $\tilde{Y}$ with respect to $\mathcal{L} = \mathcal{O}_{\tilde{M}}(\tilde{X})$.

\end{thm}

We observe P. Aluffi, in the previous theorem, assumes that $X^{'}$ is a normal crossing divisor. This hypothesis is natural, and it is guaranteed for an embedded resolution of singularities in characteristic zero. However, it is possible to obtain a normal crossing divisor with a finite sequence of the blowups. We note that in some cases the amount of blowups can be very large. In our main result, we changed the normal crossing hypothesis in divisor by a divisor whose its singular set has codimension greater than or equal to three.

These hypotheses are satisfied at the case listed below, called Nash Construct for foliations. Let $\mathcal{F}$ be a holomorphic foliation of dimension $k$ on a $n$-dimensional manifold $M$. We consider, for each $x \in M$ the following vector space

$$ F(x) = \{ v(x); v \in \mathcal{F}_{x} \},$$ 

\noindent where $\mathcal{F}_{x}$ denotes the stalk of the sheaf $\mathcal{F}$ at $x$. We note that $\dim F(x) \leq k$ and the equality is when $x \in M \setminus \Sing(\mathcal{F})$. Then, we define a section 

$$ s: M \setminus \Sing(\mathcal{F}) \longrightarrow G(k,n),$$

\noindent gives by $s(x) = F(x),$ where $G(k,n)$ denotes de Grassmannian bundle of $k$-planes in $TM$. We define $M^{\nu}$ as the closure of the $Im(s)$ in $G(k,n)$ and call it of the Nash modification of $M$ with respect to $\mathcal{F}$, see \cite{BraSu, Ser} for more details. In this context, we have that $\pi: M^{\nu} \rightarrow M$ is a proper birational map, which is induced by the projection map of the bundle $G(k,n)$, since $\pi$ is an isomorphism from $M^{\nu} \setminus \pi^{-1}(\Sing(\mathcal{F}))$ to $M \setminus \Sing(\mathcal{F}).$ In a special case, Sertoz in \cite{Ser} was studied this construction with the hypothesis that $M^{\nu}$ is a manifold, that occurs when the coherent sheaf $\mathcal{F}$ is "nice", namely, either it is gives by complex actions of reductive groups or it has locally free tangent sheaf and its singular set is smooth, see (\cite{Ser}, Corollary 1.2 p.230).

\begin{example}	Let  $M = \mathbb{P}^{n}$ be a complex projective space and $\mathcal{F}$ be an one-dimensional holomorphic foliation on $\mathbb{P}^{n}$ with isolated singularities. Consider $\Sing(\mathcal{F})=X$. Thus by (\cite{Ser}, Corollary 1.2 p.230) one has $M^{\nu} \subset \mathbb{P}^{n} \times \mathbb{P}^{n-1}$ is a nonsingular variety of dimension $n$. And $\pi^{-1}(\Sing(\mathcal{F})) = \Sing(\mathcal{F}) \times \mathbb{P}^{n-1} = \tilde{X}$ with $\tilde{Y} =\Sing(\tilde{X})= \emptyset$. So, applying the Theorem \ref{thm 1.2},  we have: 
	
	\begin{equation}
	i_{\ast}c_{SM}(X) = c(T \mathbb{P}^{n})\cap [\mathbb{P}^{n}]- \pi_{\ast} \Big( c(\Omega^{1}_{M^{\nu}}(\log \mathbb{P}^{n-1})^{\vee})\cap [M^{\nu}]\Big) .
	\end{equation}
	
	In particular, taking the degrees, we have 
	
	\begin{equation}
	 \chi(X) = \chi(\mathbb{P}^{n})- \int \pi_{\ast} \Big( c(\Omega^{1}_{M^{\nu}}(\log(\tilde{X})^{\vee})\cap [M^{\nu}]\Big)
	\end{equation}

\noi or $$\chi(\mathbb{P}^{n} \setminus X) = \int \pi_{\ast}\Big( c(\Omega^{1}_{M^{\nu}}(\log(\tilde{X})^{\vee})\cap [M^{\nu}]\Big) = \chi(M^{\nu}\setminus \tilde{X}).$$

The last equality follows from Corollary 1.2, see (\cite{Moscow} p.491). This result was expected because $\pi$ is an isomorphism from $M^{\nu} \setminus \tilde{X}$ to $\mathbb{P}^{n} \setminus X.$



\end{example}

When the singular set $\tilde{Y}$ of $\tilde{X}$  is supported at a point $P$, the $\mu$-class of $\tilde{Y}$ is $m_P[P]$, where $m_{P}$ is the classical Minor number, see (\cite{Aluffi1}, \S2). So we get the following.

\begin{corollary} In conditions of the Theorem \ref{thm 1.1} and under the additional hypothesis that $\tilde{Y}$ is supported in a set of finitely many points, i.e. $\tilde{Y} = \{x_{1}, \ldots,x_{r} \}$, then

$$i_{\ast}c_{SM}(X) =c(TM)\cap [M]- \pi_{\ast}c(\Omega^{1}_{\tilde{M}}(\log \tilde{X})^{\vee})\cap [\tilde{M}] + \pi_{\ast}j_{\ast}(-1)^{\dim \tilde{M}}\sum_{i=1}^{r} m_{i}[x_{i}],$$

\noi where $m_{i}$ is the Milnor number of $\tilde{X}$ at $x_{i}$.

\end{corollary}

\begin{proof} Under hypothesis of this corollary one has $\mu_{\mathcal{L}}(\tilde{Y}) =\sum_{i=1}^{r} m_{i}[x_{i}]$.
Now we calculate $\Big(\mu_{\mathcal{L}}(\tilde{Y})^{\vee} \otimes_{\tilde{M}}  \mathcal{L} \Big)$ by using  the definition on (\cite{Aluffi2}, p. 3996),

$$c(\mathcal{L})^{\dim \tilde{X}} \cap \Big(\mu_{\mathcal{L}}(\tilde{Y})^{\vee} \otimes_{\tilde{M}}  \mathcal{L}    \Big) =  (-1)^{\dim \tilde{M} } \sum_{i=1}^{r} m_{i}[x_{i}].$$

\end{proof}

As a consequence of the Theorem \ref{thm 1.1}, taking the degree in equation (\ref{01}), we obtain an expression for the Euler characteristic for the complement $M\setminus X$, of the variety $X$ in $M$.

\begin{corollary}\label{corollary 1.3} In conditions of Theorem \ref{thm 1.1} with the additional hypothesis that $M$ and $\tilde{M}$ are complete varieties, then

$$ \chi(M\setminus X)= \int_{\tilde{M}} c\Big(\Omega^{1}_{\tilde{M}}(\log \tilde{X})^{\vee}\Big)\cap [\tilde{M}]
+ (-1)^{n+1} \int_{\tilde{M}} \mu_{\mathcal{L}}(\tilde{Y}).$$

\end{corollary}

\begin{rmk} We observe that if  $\tilde{Y}$ is supported in a set of finitely many points, i.e. $\tilde{Y} = \{x_{1}, \ldots,x_{r} \}$, then we recover the result in (\cite{Moscow}, Corollary 1.2 (I), p.495)

$$\chi(\tilde{M}\setminus \tilde{X}) = (-1)^{n} \int_{\tilde{M}} c\Big( \Omega^{1}_{\tilde{M}} (\log \tilde{X})\Big)\cap[\tilde{M}] + (-1)^{n+1}\sum_{i=1}^{r} m_{i}, $$

\noi since $\chi(M\setminus X) = \chi(\tilde{M}\setminus \tilde{X})$ and $\displaystyle \int_{\tilde{M}} \mu_{\mathcal{L}}(\tilde{Y}) = \sum_{i=1}^{r} m_{i}.$
\end{rmk}

In the second part of this work we use the sheaves of multi-logarithmic differential forms associated to a complete intersection, see Section $4$ and \cite{AleTs, Ale_s} for more details about these sheaves. We prove the following  result involved its characteristic classes:

\begin{thm}\label{multlog} Let $i: X \longrightarrow M$ be an embedding of a closed subvariety $X$ in a complex algebraic nonsingular variety $M$ and $\pi :\tilde{M} \longrightarrow M$ be a proper birational map with $\tilde{M}$ a nonsingular variety, such that $(\pi^{-1}(X))_{red}= C = D_{1}\cap D_{2}$ is a complete intersection of smooth divisors and $\pi|_{\tilde{M}\setminus C}$ is an isomorphism. We assume that $\tilde{D} = D_{1} \cup D_{2}$ is a normal crossing divisor. Then

$$
i_{\ast}c_{SM}(X) = c(TM)\cap [M]- \pi_{\ast} \Big( c(\Omega^{1}_{\tilde{M}}(\log C)^{\vee})\cap [\tilde{M}]\Big) + \pi_{\ast}j_{\ast}c_{SM}(C) - \pi_{\ast}j_{\ast}c_{SM}(\tilde{D}),
$$
\noi where $\Omega^{1}_{\tilde{M}}(\log C)$ denotes the sheaf of multi-logarithmic differential 1-forms on $\tilde{M}$.

\end{thm}




















The paper is organized as follows. First, in order to make this work as self-contained as
possible, we provide some necessary definitions and considerations in Sections 2, 3 and 4. The proofs
of our main results appear in Sections 5, 6 and 7.

\section*{Acknowledgements}
The authors were partially supported by the FAPEMIG [grant number 38155289/2021] and FAPEMIG  RED-00133-21.

\section{Chern Classes of Singular Varieties}

Let $X$ be a complex algebraic variety and $A$ be a subvariety of $X$, we denote by $\mathbb{1}_{A}$ the characteristic function of $A$ which is constant equal to $1$ over $A$ and constant equal to $0$ elsewhere. A constructible function on $X$ is an integral linear combination of characteristic functions of closed subvarieties of $X$. The constructible functions on $X$ forms a group denoted by $\textbf{F}(X)$. It can be made a covariant functor in the following way: for a proper morphism $f: X\rightarrow Y$ the push-forward $f_{\ast}$ is defined by setting 
$$ f_{\ast}(\mathbb{1}_{A})(y)=\chi\big( {f^{-1}}(y)\cap A\big), $$ where $A$ is a subvariety of $X$ and extending by linearity.

It was conjectured by Deligne and Grothendieck in 1969 and proved by R. MacPherson \cite{MacPherson} in 1974, that there exists a natural  transformation $c_{\ast}$ from the functor $\bold{F}$ to homology, which, on a nonsingular variety $V$, assigns to the function $\mathbb{1}_{V}$  the Poincar\'e dual of the total Chern class of $V$. That is, 
$$c_{\ast}(\mathbb{1}_{V})= c(TV)\cap V.$$ So it is natural to consider the class $c_{\ast}(\mathbb{1}_{X})$ to arbitrary variety $X$. This class is denoted by $c_{SM}(X)$ and called of \textit{Schwartz-MacPherson class} of $X$. This class is the image, via Alexander duality isomorphism, of the previously class defined by M.-H. Schwartz in 1965 see \cite{BrSch}. Explicitly, MacPherson has proved that for all constructible functions $\alpha$, $\beta$, and proper morphism $f$, the class $c_{\ast}$ satisfy the conditions:

\begin{enumerate}

\item[(i)] $f_{\ast} c_{\ast}(\alpha) = c_{\ast}f_{\ast}(\alpha);  $

\item[(ii)] $c_{\ast}(\alpha + \beta) = c_{\ast}(\alpha) + c_{\ast}(\beta); $

\item[(iii)] $c_{\ast}(\mathbb{1}_{X}) = c(TX)\cap [X], \ \ if \ \ X \ \ \mbox{is} \ \ \mbox{smooth}.$
\end{enumerate}

Although MacPherson has initially considered complex algebraic varieties, Kennedy \cite{Kennedy} indicated  how MacPherson's theory can be made completely algebraic, by extending it to varieties over an arbitrary field $k$ of characteristic zero.

Let $X$ be now a scheme which can be embedded as a closed subscheme of a nonsingular variety $M$. Then the \textsl{Fulton class} of $X$, denoted by $c_{F}(X)\in A_{\ast}(X) $, was defined by W. Fulton 1984, see (\cite{Fulton}, Example 4.2.6) by setting 

$$ c_{F}(X)= c(TM_{|X})\cap s(X,M),$$ 

\noi where $s(X,M)$ denotes the Segre class of $X$ on $M$. It is independent of the choice of embedding. 

In particular case, when $X$ is a divisor, the Segre class of $X$ is given by $s(X,M)=\frac{[X]}{1+X}$, where by abuse of notation,  we denote by $X$ the first Chern class $c_1(\mathcal{O}(X) )$. So in this case 
\begin{equation} \label{eqFu} c_{F}(X)= c(TM_{|X})\cap  \dfrac{[X]}{1+X}. \end{equation} 

Another important characteristic class, the \textsl{$\mu$-class}, was introduced by P. Aluffi \cite{Aluffi1} in 1995. Let $Y$ be the singular scheme of a hypersurface $X$ on a smooth variety M,  and let  $\mathcal{L}=\mathcal{O}(X)$ be the line bundle associate to $X$. The \textsl{$\mu$-class} of $Y$ with respect to $\mathcal{L}$  is the class
$$\mu_{\mathcal{L}}(Y):=c(T^{\ast}M \otimes \mathcal{L})\cap s(Y,M)$$ in the Chow group  $A_{\ast}(Y)$. 

In particular, when $Y$ is supported at a point $P$, then $\mu(Y)=m_{P}[P]$, 
where $m_{P}$ is the Milnor number of $X$ at $P$ see (\cite{Aluffi1}, \S2).
In the following theorem, P. Aluffi has established an interesting  relationship between these classes.

\begin{thm}[Aluffi \cite{Aluffi2}, Theorem I.5]\label{thm2.1} Let $X$ be a hypersurface in a nonsingular variety $M$, let Y be its
singular scheme, and let $\mathcal{L}=\mathcal{O}(X)$. Then 

\begin{equation}\label{eq002}
c_{SM}(X)=c_{F}(X) + c(\mathcal{L})^{\dim X}\cap(\mu_{\mathcal{L}}(Y)^{\vee} \otimes_M \mathcal{L}).
\end{equation}

\end{thm}

\section{The sheaf of logarithmic forms}

Let $M$ be a complex manifold of dimension $n$ and $X$ a reduced hypersurface on $M$. We consider $\Omega^q_{M}(X)$ the sheaf of differential $q$-forms on $M$ with at most simple poles along $ X$. We define a {\it logarithmic $q$-form along} $ X$ on an open subset $U \subset M$ by a meromorphic $q$-form $\omega$ on $U$, regular on the complement $U -  X$ and such that both forms $\omega$ and $d\omega$ are in the sheaf $\Omega^q_{M}(X)$.

Logarithmic $q$-forms along $X$ form a coherent sheaf of $\OO_M$-modules that we will denote simply by $\Omega_{M}^q(\log\,  X)$.  In this case, for any open subset $U\subset M$ we have 
$$
\Gamma(U,\Omega_{M}^q(\log\,  X)) = \{\omega \in \Gamma(U,\Omega^q_{M}( X)): d\omega \in \Gamma(U,\Omega^{q+1}_{M}( X))\}.
$$
See for example \cite{Alex1}, \cite{Dolga}   and  \cite{Sai}  for more details about the sheaf of logarithmic $q$-forms along $X$.

Take  $\Omega_{M}^1(\log\,  X)$, the sheaf of logarithmic $1$-forms  along $ X$. Its dual is the sheaf of logarithmic vector fields along $ X$, denoted by $T_{M}(-\log\,  X)$ or $Der(-\log X)$. With this notations we have the classical short exact sequence (see \cite{Dolga} \S 2) 
$$
\xymatrix{
0\ar[r]& T_{M}(-\log\,  X) \ar[r]& T_M \ar[r] &  \ar[r]  J_{X}(X)&0,  \\
}$$

\noindent where $J_{X}$ denotes the Jacobian ideal of $X$ which  is defined as the Fitting ideal
$$
J_{X}:=F^{n-1}(\Omega_X^1) \subset \mathcal{O}_X. 
$$


\hyphenation{theo-ry}

Saito in \cite{Sai} has showed that in general  $\Omega_{M}^1(\log\,  X)$ and $T_{M}(-\log \, X)$  are  reflexive sheaves. 
When  $ X$ is an analytic hypersurface with normal crossing singularities, the sheaves $\Omega_{M}^1(\log\,  X)$ and $T_{M}(-\log \, X)$ are  locally free. Furthermore, the Poincar\'e residue map (see \cite[Section 2]{Sai})

\centerline{
\label{seq1} \xymatrix{
\rm{Res}: \Omega_M^1(\log\,  X) \ar[r]  &  \OO_{ X} \cong \bigoplus_{i=1}^N\OO_{ X_i}  \\
}
}

\noindent  gives us  the following exact sequence of sheaves on $M$: 

\begin{eqnarray}\label{seq1}
\centerline{
\xymatrix{
0\ar[r]& \Omega_M^1 \ar[r]&   \Omega_M^1(\log\,  X)\ar[r]^{\text{Res}} &  \ar[r]   \bigoplus_{i=1}^N\OO_{ X_i}&0,  \\
}
}
\end{eqnarray}
\noindent where $\Omega_M^1$ is the sheaf of holomorphic  $1$-forms on $M$ and $ X_1,\ldots, X_N$ are the irreducible components of $ X$.

Finally, if  $ X$ is such that $\codim_{M}(\Sing( X)\geq 3$, then there exist the following exact sequence of sheaves on $M$ (see  \cite{Dolga}):
\begin{eqnarray} \label{seq0002}
\centerline{
\xymatrix{
0\ar[r]& \Omega_M^1 \ar[r]&   \Omega_M^1(\log\,  X)\ar[r]  &    \OO_{ X} \ar[r] &0.  \\
}}
\end{eqnarray}

\section{The sheaf of multi-logarithmic forms} In this section we will give some basic definitions on the theory of multi-logarithmic differential forms and the results that we will need in this paper. For more details and properties on this subject see \cite{Ale_s} and \cite{POL}. 

Let $X= X_{1} \cup \cdots \cup X_{k}$ be a decomposition of the reduced hypersurface $X$ in a complex manifold $M$, where each $X_{i}$ is a hypersurface defined by a holomorphic function $h_{i}$, for $i = 1, \dots ,k$, on an open subset $U \subset M$, and $C = X_{1} \cap \cdots \cap X_{k}$ is a reduced complete intersection.

We define a {\it multi-logarithmic $q$-form along} the complete intersection $C$ on an open subset $U \subset X$ by a meromorphic $q$-form $\omega$ on $U$, regular on $U - X$ and such that $\omega$ have at most simple poles along $C$ and 
$$dh_{i}\wedge\omega  \in \sum_{i=1}^{k}\Omega^{q+1}(\widehat{X_{i}}) \ \ \ \mbox{for all} \ \ i \in \{1, \dots,k \}, $$

\noindent where $\widehat{X_{i}}= X_{1} \cup \cdots \cup X_{i-1}\cup X_{i+1}\cup \cdots \cup X_{k}.$

We denote by $\Omega^{q}_{M}(\log C)$ the coherent sheaf of germs of multi-logarithmic $q$-forms along $C$. A. G. Aleksandrov \cite{Ale_s} has proved the following result that characterizes multi-logarithmic forms.

\begin{thm}[A. G. Aleksandrov, \cite{Ale_s}]

Let $\omega \in \Omega^{q}_{M}(X)$, then $\omega$ is multi-logarithmic along $C$ if and only if there is a holomorphic function $g \in \mathcal{O}_{M}$ which is not identically zero on every irreducible component of the $C$, a holomorphic differential form $\xi \in \Omega^{q-k}_{M}$ and a meromorphic $q$-form $\eta \in \sum_{i=i}^{k} \Omega^{q}_{M}(\widehat{X_{i}})$ such that there exists the following representation

$$g\omega = \dfrac{dh_{1}\wedge \cdots \wedge dh_{k}}{h_{1} \cdots h_{k}}\wedge \xi + \eta.$$
\\ 
\end{thm}

For $q<k$ we have the equality (see \cite{POL}, Remark 2.6): 

$$ \Omega^{q}_{M}(\log C)= \sum_{i=i}^{k} \Omega^{q}_{M}(\widehat{X_{i}}).$$

\noindent Observe that if $q=1$ and $k=2$ we have

\begin{equation}\label{dd01}
\Omega^{1}_{M}(\log C)= \Omega^{1}_{M}(X_{1}) + \Omega^{1}_{M}(X_{2}). 
\end{equation}



\begin{proposition}[see \cite{Moscow}]\label{teo 2.4} Let $X= X_{1} \cup X_{2}$ be a reduced hypersurface on   $M$, where $X_{i}$ is a reduced hypersurface, for $i=1,2,$ and $C= X_{1}\cap X_{2}$ is a reduced  complete  intersection. Then
$$\Omega^{1}_M(\log X) = \Omega^{1}_M(\log X_{1}) + \Omega^{1}_M(\log X_{2}).$$
\end{proposition}

\begin{lemma}\label{basic}
Let $\tilde{D} = D_{1} \cup D_{2}$ be  a normal crossing divisor, where $D_{1}, D_{2}$ be smooth hypersurfaces in a complex manifold $\tilde{M}$, such that $C = D_{1} \cap D_{2}$ is a complete intersection. Then
$$c(\Omega_{\tilde{M}}^{1}(\log C)) = c(\Omega_{\tilde{M}}^{1}(\log \tilde{D})).$$

\end{lemma} 

\begin{proof} In fact, by the following exact sequence
$$0 \longrightarrow \Omega_{X}^{1} \longrightarrow \Omega_{X}^{1}(D_{1}) \oplus \Omega_{\tilde{M}}^{1}(D_{2}) \longrightarrow \Omega_{\tilde{M}}^{1}(D_{1}) + \Omega_{X}^{1}(D_{2}) \longrightarrow 0$$
\noi one has by Chern class properties,
$$c\Big( \Omega_{X}^{1}(D_{1}) \oplus \Omega_{X}^{1}(D_{2}) \Big) = c\Big( \Omega_{X}^{1}(D_{1}\Big) c \Big( \Omega_{X}^{1}(D_{2}) \Big) =  c\Big( \Omega_{X}^{1}(D_{1}) + \Omega_{X}^{1}(D_{2}) \Big)c(\Omega_{X}^{1}).$$ 

\noi Let us consider the exact sequence
$$0 \longrightarrow \Omega_{X}^{1} \longrightarrow \Omega_{X}^{1}(D_{i}) \longrightarrow \mathcal{O}_{D_{i}} \longrightarrow 0.$$
\noi then, one has \ \ $c(\Omega_{X}^{1}(D_{i})) = c(\Omega_{X}^{1})c(\mathcal{O}_{D_{i}}),$ \ \ \ where $c_{j}(\mathcal{O}_{D_{i}}) = c_{1}([D_{i}])^{j}.$ \\\

\noi Then $c\Big( \Omega_{X}^{1}(D_{1}) + \Omega_{X}^{1}(D_{2}) \Big) = c(\mathcal{O}_{D_{1}}) c(\mathcal{O}_{D_{2}})c(\Omega_{X}^{1})$. Since $\Omega_{X}^{1}(\log C) = \Omega_{X}(D_{1}) + \Omega_{X}(D_{2}),$ we get

$$c(\Omega_{X}^{1}(\log C)) = c(\mathcal{O}_{D_{1}})c(\mathcal{O}_{D_{2}})c(\Omega_{X}).$$
On the other hand, since $\tilde{D} = D_{1}\cup D_{2}$ is a normal crossing divisor, we have the following exact sequence, see sequence (\ref{seq1}),
$$0 \longrightarrow \Omega_{\tilde{M}}^{1} \longrightarrow \Omega_{\tilde{M}}^{1}(\log(D_{1} \cup D_{2})) \longrightarrow { \mathcal{O}_{D_1} \oplus \mathcal{O}_{D_{2}}} \longrightarrow 0.$$

It is follows that $c(\Omega_{\tilde{M}}^{1}(\log(D_{1} \cup D_{2}))) = c(\mathcal{O}_{D_{1}})c(\mathcal{O}_{D_{2}})c(\Omega_{X})$  and we finish the prove.
\end{proof}
\hspace{0.1cm}
\section{Proof of Theorem \ref{thm 1.1} }

\begin{proof} 



Let us consider the exact sequence (\ref{seq0002})  of reflexive sheaves on $\tilde{M}$ 

$$ 0\longrightarrow \Omega^{1}_{\tilde{M}} \longrightarrow \Omega^{1}_{\tilde{M}}(\log \tilde{X}) \longrightarrow \mathcal{O}_{\tilde{X}} \longrightarrow 0$$

So we have

\begin{equation}\label{eq02} c(\Omega^{1}_{\tilde{M}}(\log \tilde{X}))= c(\Omega^{1}_{\tilde{M}})c(\mathcal{O}_{\tilde{X}})\end{equation}

Now let us recall the exact sequence, see (\cite{Daniel}, p.84).
$$ 0 \longrightarrow \mathcal{O}(-\tilde{X}) \longrightarrow  \mathcal{O}_{\tilde{M}} \longrightarrow \mathcal{O}_{\tilde{X}} \longrightarrow 0$$

So

\begin{equation}\label{eq03} c(\mathcal{O}_{\tilde{X}}) = \dfrac{1}{ c(\mathcal{O}(-\tilde{X}))}\end{equation}

Then from (\ref{eq02} ) and (\ref{eq03})

$$c(\Omega^{1}_{\tilde{M}}(\log \tilde{X}))= \dfrac{c(\Omega^{1}_{\tilde{M}})}{c(\mathcal{O}(-\tilde{X}))}$$

We take the dual in previous expression

$$c(\Omega^{1}_{\tilde{M}}(\log \tilde{X})^{\vee})= \dfrac{c(T \tilde{M})}{c(\mathcal{O}(\tilde{X}))} = \dfrac{c(T \tilde{M})}{1 + c_{1}(\mathcal{O}(\tilde{X}))} = \dfrac{c(T \tilde{M})}{1 + \tilde{X}} $$

Let us consider the following calculation 

$$\begin{array}{ccl} 
c(T\tilde{M})\Big(1- \dfrac{1}{1 + \tilde{X}}\Big) \cap [\tilde{M}] &=& 
c(T\tilde{M}) \Big(\dfrac{\tilde{X}}{1 + \tilde{X}}\Big)\cap [\tilde{M}] \\ &=& 
j_{\ast}c(T\tilde{M}) \cap \Big(\dfrac{[\tilde{X}]}{1 + \tilde{X}}\Big ) \\ &=& j_{\ast}c_{F}(\tilde{X}),
\end{array}
$$

\noindent where $j : \tilde{X} \longrightarrow \tilde{M}$ is the inclusion morphism. 

On the other hand,
\begin{equation}\begin{array}{rll}\label{eq05}

j_{\ast}c_{F}(\tilde{X})= & c(T\tilde{M})\Big(1- \dfrac{1}{1 + \tilde{X}}\Big) \cap [\tilde{M}]= c(T\tilde{M})\cap [\tilde{M}] - \dfrac{c(T\tilde{M})}{1 + \tilde{X}}\cap [\tilde{M}] \\ \\

= & c(T\tilde{M})\cap[\tilde{M}]- c(\Omega^{1}_{\tilde{M}}(\log \tilde{X})^{\vee})\cap [\tilde{M}].

\end{array}\end{equation}





Using the expression in Theorem \ref{thm2.1} to $c_{F}(\tilde{X})$  and applying the homomorphism $j_{\ast}$, one has
$$j_{\ast}c_{F}(\tilde{X}) = j_{\ast}c_{SM}(\tilde{X})  - j_{\ast}c(\mathcal{L})^{\dim \tilde{X}} \cap \Big(\mu_{\mathcal{L}}(\tilde{Y})^{\vee} \otimes_{\tilde{M}} \mathcal{L}    \Big).$$
Now, we use the above expression in equation (\ref{eq05})
$$j_{\ast}c_{SM}(\tilde{X})  - j_{\ast}c(\mathcal{L})^{\dim \tilde{X}} \cap \Big(\mu_{\mathcal{L}}(Y)^{\vee} \otimes_{\tilde{M}} \mathcal{L}    \Big) = c(T\tilde{M})\cap[\tilde{M}]- c(\Omega^{1}_{\tilde{M}}(\log \tilde{X})^{\vee})\cap [\tilde{M}].$$
$$ 
c(\Omega^{1}_{\tilde{M}}(\log \tilde{X})^{\vee})\cap [\tilde{M}] - j_{\ast}c(\mathcal{L})^{\dim \tilde{X}} \cap \Big(\mu_{\mathcal{L}}(\tilde{Y})^{\vee} \otimes_{\tilde{M}} \mathcal{L}    \Big)  = c_{SM}(\tilde{M}) - j_{\ast}c_{SM}(\tilde{X})$$ 
\begin{equation}\label{eq45}
c(\Omega^{1}_{\tilde{M}}(\log \tilde{X})^{\vee})\cap [\tilde{M}] - j_{\ast}c(\mathcal{L})^{\dim \tilde{X}} \cap \Big(\mu_{\mathcal{L}}(\tilde{Y})^{\vee} \otimes_{\tilde{M}} \mathcal{L}    \Big)  = c_{\ast}(\mathbb{1}_{\tilde{M}\setminus \tilde{X}}).  
\end{equation}
 
Using Schwartz-MacPherson class properties one has,
$$\pi_{\ast}c_{\ast}(\mathbb{1}_{\tilde{M}\setminus \tilde{X}}) =
c_{\ast}\pi_{\ast}(\mathbb{1}_{\tilde{M}\setminus \tilde{X}}) = c_{\ast}(\mathbb{1}_{M\setminus X}) = c(TM)\cap [M] - i_{\ast}c_{SM}(X).$$







So, applying $\pi_{\ast}$ in the equation (\ref{eq45}), we have
$$i_{\ast}c_{SM}(X) =c(TM)\cap [M]- \pi_{\ast}c(\Omega^{1}_{\tilde{M}}(\log \tilde{X})^{\vee})\cap [\tilde{M}] + \pi_{\ast}j_{\ast}c(\mathcal{L})^{\dim \tilde{X}} \cap \Big(\mu_{\mathcal{L}}(\tilde{Y})^{\vee} \otimes_{\tilde{M}} \mathcal{L}    \Big).$$
\end{proof}

\section{Proof of Corollary \ref{corollary 1.3} }

\begin{proof}

We note by Schwartz-MacPherson class properties one has 
 
$$ \pi_{\ast} c_{SM}({ \tilde{M}\setminus{\tilde{X}}}) = c_{SM}(M\setminus X).$$

So, by using the degree properties (see \cite{Fulton}, p. 13) we have

$$
 \int_{ \tilde{M}} c_{SM}({ \tilde{M}\setminus{\tilde{X}}}) =   \int_{M }  \pi_{\ast}c_{SM}( \tilde{M}\setminus{\tilde{X}})  =\int_M c_{SM}(M\setminus X) = \chi(M \setminus X).
$$
Taking the degrees in equation (\ref{eq45})
 $$ \begin{array}{lll}  \chi(M\setminus X) &=&  \displaystyle \int_{\tilde{M}} c(\Omega^{1}  _{\tilde{M}}(\log \tilde{X})^{\vee})\cap [\tilde{M}]
- \int_{\tilde{M}} j_{\ast}c(\mathcal{L})^{\dim \tilde{X}} \cap \Big(\mu_{\mathcal{L}}(\tilde{Y})^{\vee} \otimes \mathcal{L}    \Big)\\  \\ 

 &=& \displaystyle \int_{\tilde{M}} c(\Omega^{1}_{\tilde{M}}(\log \tilde{X})^{\vee})\cap [\tilde{M}]
+ (-1)^{n+1} \int_{\tilde{M}} \mu_{\mathcal{L}}(\tilde{Y}).
\\ \\ 

\end{array} 
$$

The last simplification follows from Aluffi (see \cite{Aluffi2}, \S 4).
\end{proof}

\section{Proof of Theorem \ref{multlog}}

\begin{proof} How $\tilde{D} = D_{1} \cup D_{2}$ is a normal crossing divisor,  we have the short exact sequence (\ref{seq1})  

$$0 \longrightarrow \Omega_{\tilde{M}}^{1} \longrightarrow \Omega_{\tilde{M}}^{1}(\log(D_{1} \cup D_{2})) \longrightarrow {D_1 \oplus D_2}  \longrightarrow 0$$

Following the proof of Theorem 1 in (\cite{Aluffi}, p.621) we have

$$j_{\ast}c_{SM}(\tilde{D}) = c(T\tilde{M}) \cap [\tilde{M}] - c(\Omega^{1}_{\tilde{M}}(\log \tilde{D})^{\vee})\cap [\tilde{M}].$$

Now, adding the factor $(-j_{\ast}c_{SM}(C))$ in both sides and applying the morphism $\pi_{\ast}$

$$\begin{array}{ccl}

j_{\ast}c_{SM}(\tilde{D}) - j_{\ast}c_{SM}(C)= c(T\tilde{M}) \cap [\tilde{M}] - c(\Omega^{1}_{\tilde{M}}(\log \tilde{D})^{\vee})\cap [\tilde{M}] - j_{\ast}c_{SM}(C) \\ \\

j_{\ast}c_{SM}(\tilde{D}) - j_{\ast}c_{SM}(C)= c_{\ast}(\mathbb{1}_{\tilde{M}\setminus \ C}) - c(\Omega^{1}_{\tilde{M}}(\log \tilde{D})^{\vee})\cap [\tilde{M}] \\ \\

\pi_{\ast}j_{\ast}c_{SM}(\tilde{D}) - \pi_{\ast}j_{\ast}c_{SM}(C) =c_{\ast}(\mathbb{1}_{M \setminus \ X}) - \pi_{\ast}c(\Omega^{1}_{\tilde{M}}(\log \tilde{D})^{\vee})\cap [\tilde{M}].

\end{array}$$

To finish the proof we use the Lemma \ref{basic}
$$\pi_{\ast}j_{\ast}c_{SM}(\tilde{D}) - \pi_{\ast}j_{\ast}c_{SM}(C) =  c_{\ast}(\mathbb{1}_{M \setminus \ X}) - \pi_{\ast}c(\Omega^{1}_{\tilde{M}}(\log C)^{\vee})\cap [\tilde{M}].$$

Rearranging it
$$i_{\ast}c_{SM}(X) = c(TM)\cap [M]- \pi_{\ast} \Big( c(\Omega^{1}_{\tilde{M}}(\log C)^{\vee})\cap [\tilde{M}]\Big) + \pi_{\ast}j_{\ast}c_{SM}(C) - \pi_{\ast}j_{\ast}c_{SM}(\tilde{D}),
$$
\end{proof}

\end{document}